\documentclass[oneside]{amsart}
\usepackage{amsfonts}
\usepackage{amssymb}
\usepackage{amsxtra}
\usepackage{amstext}
\usepackage[english]{babel}

\newcommand{\N}{\mathbb{N}}

\newcommand{\Z}{\mathbb{Z}}

\newcommand{\pp}{\mathbb{P}}

\newcommand{\kO}{\mathcal{O}}

\newcommand{\kF}{\mathcal{F}}

\newtheorem {lem} {Lemma} [section]

\newtheorem {theo} {Theorem} [section]
\newtheorem {cor} {Corollary} [section]

\title[A $0-1$ law for VRRW  with weight of order $k^\alpha$.]
      {A $0-1$ law for vertex-reinforced random walks on $\Z$ with weight of order $k^\alpha$, $\alpha<1/2$.}

\author{Bruno SCHAPIRA}
\address{D\'epartement de Math\'ematiques, B\^at. 425, Universit\'e Paris-Sud 11, F-91405 Orsay, cedex, France. }
\email{bruno.schapira@math.u-psud.fr}

\begin{document}

\begin{abstract} We prove that Vertex Reinforced Random Walk on $\Z$ with weight 
of order $k^\alpha$, with $\alpha\in [0,1/2)$, is either almost surely recurrent or almost surely transient. 
This improves a previous result of Volkov who showed that the set of sites which are visited infinitely often was a.s. either empty or infinite. 
\end{abstract} 

\maketitle

\section{Introduction}
Linearly vertex reinforced random walks (VRRW for short), introduced by Pemantle in \cite{P}, were first studied on $\Z$ by Pemantle and Volkov \cite{PV},  
who showed that with positive probability these processes spend all large times on just five sites. 
Some times later, Tarr\`es \cite{T} managed to prove that this striking phenomenon, called localization, occurs in fact almost surely 
(and he recently gave a simplified proof in \cite{T2}). 
Roughly in the mean-time Volkov \cite{V1} proved that (linearly) VRRW localize as well on a 
large class of graphs with positive probability and almost surely on trees. Bena\"im and Tarr\`es \cite{BT} have recently 
generalized his result to a larger class of walks, with a completely different proof.

In the recent works \cite{ETW1, ETW2}, new models of self-interacting random walks are introduced, 
where the interaction is not restricted to nearest neighbors. Then the authors prove that localization can occur on sets of arbitrary size, depending on the 
parameters of the model.

What emerges from these remarkable results is the fact that, when studying self-interacting (or non Markovian) 
random walks on graphs, the first thing one should do is to determine 
the set of vertices which are visited infinitely often and see if this is empty, the whole graph or some nontrivial subgraph.  According to Volkov's notation \cite{V2} we shall denote this set by $R'$ here. 
If it is the whole graph, we say that the walk is recurrent 
and if it is empty we say that the walk is transient. But as we just saw, it might be equal to something else and even have arbitrary size.

In 2006 Volkov started the study of VRRW on $\Z$ with some weight $(w_k,k\ge 0)$. 
Such process, say $(X_n,n\ge 0)$, is defined as follows. First $X_0=0$. Then for all $n\ge 0$, on the event $\{X_n=x\}$,
\begin{eqnarray}
\label{loiX}
\pp(X_{n+1} = X_n \pm 1\mid \kF_n) = \frac{w_{Z_n(x\pm 1)}}{w_{Z_n(x-1)}+w_{Z_n(x+1)}},
\end{eqnarray}
where $(\kF_n,n\ge 0)$ is the natural filtration of $X$ and for all $y\in \Z$, 
$$Z_n(y) = \#\{ m\le n \ :\ X_m = y\},$$
is the local time in $y$ at time $n$. Linearly VRRW correspond to the case when $w_k=k+1$ for all $k\ge 0$. 
Volkov \cite{V2} showed that when $\sum_k 1/w_k$ is finite, then $X$ almost surely localizes on two sites, 
i.e. that $R'$ has a.s. cardinality $2$, and when $w_k \asymp (k+1)^{\alpha}$,\footnote{we say that $f_k\asymp g_k$ when 
$f_k/g_k$ is bounded from above and below by positive constants} for some $\alpha\in [0,1)$, then a.s. $R'$ cannot be nonempty and finite (actually he
proved this result under slightly more general hypotheses, see \cite{V2} for details). 
He conjectured also that the process should be a.s. recurrent, i.e. that $R'$ should be a.s. equal to $\Z$  (Problem 3 in \cite{V2}).  
It is even natural to believe that this should hold as soon as 
$w_k=\kO(k^\alpha)$, for some $\alpha<1$, and not only when $w_k$ is exactly of order $k^\alpha$, for some fixed $\alpha$. 
However, to our knowledge, no progress on this conjecture has been made since then, even in the case $\alpha=0$. Here we obtain the following result:
\begin{theo} 
\label{theo1} 
Assume that there exists some $\alpha\in [0,1/2)$, such that $w_k\asymp (k+1)^\alpha$. 
Then the VRRW on $\Z$ with weight  
$(w_k,k\ge 0)$ is either a.s. recurrent or a.s. transient. 
\end{theo} 
This result says that $R'$ is either a.s. empty or a.s. equal to $\Z$. This first step toward Volkov's conjecture, called Problem 1 in \cite{V2}, 
gives strong evidence that the conjecture should be true, at least when $\alpha<1/2$. 
Indeed since the process is "reinforced" it should be "more" recurrent than simple random walk and since it does not localize it should be recurrent. 
However, giving a rigorous proof to this kind of monotonicity argument (even formulating a correct statement) is still a real challenge.

For other results on VRRW, particularly on finite graphs, we refer the reader to \cite{B, BT, LV, P}. 
We shall also mention that analogous results have been obtained in a continuous setting, 
for self-interacting diffusions, see \cite{CLJ, HR, R}.

Our proof is different from Volkov's proof, which was based on urns arguments and on Rubin's construction.
We use instead a kind of domino principle, which works roughly as follows. Assume that 
some site $x\le 0$ is visited infinitely often, but not $x-1$, and let us fix some small constant $\epsilon>0$. 
Then at $k$-th visit to $x$, with $k$ large, 
the local time in $x+1$ has to be at least of order $k^{1/\alpha-\epsilon}$. Otherwise, 
$X$ would have jumped roughly $k^{\alpha \epsilon}$ times on $x-1$, which is not allowed if $k$ is large. 
Then we repeat this argument and show
that before the $k^{1/\alpha-\epsilon}$-th visit to $x+1$, the local time in $x+2$ has to be at least of order $k^\gamma$, 
with $\gamma =1/\alpha +(1/\alpha-1)^2-\epsilon/\alpha$. Otherwise, during the $k^{1/\alpha-\epsilon}$ visits to $x+1$, 
$X$ would have jumped more than $k$ times to $x$. By repeating this argument infinitely often, we get 
that the local time in $x+i$ has to be of order $k^{\gamma_i}$, with $\gamma_i$ of order $(1/\alpha-1)^i$, for all $i\ge 1$. 
This is of course not possible before the time of $k$-th visit to $x$, and we get a contradiction. However this argument only works when 
$\gamma_i \to \infty$, when $i\to \infty$, which explains why we need the hypothesis $\alpha<1/2$.   
Then we deduce that a.s. $R'$ is either empty or equal to $\Z$, see Sections \ref{secinduction} and \ref{alpha<1/2} for more details. 
To see that there is really a $0-1$ law, we use the general Lemma \ref{gene1} below, which enables us to conclude with Borel-Cantelli like arguments.

\section{A general lemma} 
Let us introduce some new notation. For any $w=(w_k(x))_{x\in \Z,k\ge 0}$, denote by $\pp_w$ 
the law of the VRRW in the "environement" $w$. This process is defined as in \eqref{loiX} except that in the right hand side we replace 
$w_{Z_n(x\pm 1)}$ by $w_{Z_n(x\pm 1)}(x\pm 1)$. 
\begin{lem} 
\label{gene1} If $0$ has positive probability under $\pp$ to be visited only finitely many times, then for all $w\in (0,\infty)^{\Z\times \N}$, such that  
$w_k(x) = w_k$ for all $x\ge 0$ and $k\ge 0$, the probability under $\pp_w$ that $0$ is visited only at time $0$ is also positive. In particular $\pp(X_n>0 \ \forall n>0)>0$. 
\end{lem} 
\begin{proof} By using a symmetry argument we know that there exists some $M>0$ such that $\pp(X_n>0 \textrm{ for all }n\ge M) >0$.  
By conditioning now with respect to the first $M$ steps, we see that there must exist some sequence $(x_0,\dots,x_M)$, with $x_M>0$, such that conditionally on $E=\{(X_0,\dots,X_M)=(x_0,\dots,x_M)\}$, the probability that $X_n>0$ for all $n>M$ is positive. But for any such sequence and any $w$ as in the lemma, we have $\pp_w(E)>0$, and since $w_k(x) = w_k$ when $x\ge 0$, we have 
$$\pp_w(X_n>0 \textrm{ for all }n>M \mid E) = \pp(X_n>0 \textrm{ for all } n>M \mid E)>0.$$
Note that if $X$ follows the path $(x_0,\dots,x_M)$ during the first $M$ 
steps and after stays on the right of $0$, then certainly it always stays on the right of $-M$. 
Thus we also have 
\begin{eqnarray}
\label{infM}
\pp_w(X_n > -M \textrm{ for all }n\ge 0) >0.
\end{eqnarray}
Now if $w'=(w'_k(x))_{k\ge 0,x\in \Z}$ is such that for all $k\ge 0$, $w'_k(x) = w_{k+1}$ if $x<0$, and $w'_k(x)=w_k$ if $x\ge 0$, then by using the Markov property we get 
$$
\pp_w(X_n>0\ \forall n>0)\ge \pp_w(X_1=1,\dots,X_M=M)\pp_{w'}(X_n > -M\ \forall n\ge 0).$$
The first probability on the right hand side is positive (since $w_0>0$ by hypothesis), and it follows from \eqref{infM}, with $w'$ in place of $w$, that the second one is also positive. This finishes the proof of the lemma. 
\end{proof}

\section{An induction argument and a new proof of Volkov's result}
\label{secinduction}
We first present a kind of domino principle. 
In plain words it works as follows.  
Assume that there exists some $x\in \Z$, such that $\inf R'=x$. It means that $x$ is visited infinitely often, but not $x-1$. 
To simplify assume even that $x-1$ has never been visited. 
Fix some large integer $k$ and let $n$ be the time of $k$-th visit to $x$. 
Then at each of the $k$ first visits to $x$, the process has probability at least of order $1/Z_n(x+1)^\alpha$ to jump to $x-1$. 
Since it did not, this implies with high probability that $Z_n(x+1)$ is at least of order $k^{1/\alpha}$. 
The idea is then to repeat the argument. More precisely the next lemma implies by induction that the local time in $x+i$ at time $n$ 
is of order at least $k^{\gamma_i}$, with $\gamma_i=\sum_{j=0}^{i}(1/\alpha-1)^j$, up to some error term and with probability going 
to $1$ exponentially fast when $k\to \infty$. In particular when $\alpha < 1/2$ the error term is negligible and we get a contradiction, 
since the process $X$ cannot visit an infinite number of sites before time $n$. See the next subsection for details. 
When $\alpha$ is larger than or equal to $1/2$, the error term becomes predominant when $i\to \infty$, and the argument blows up.  
However, it still implies that $R'$ cannot be finite, which gives an alternative proof to Volkov's result, see Corollary \ref{R'infini} below.

\vspace{0.2cm} 
\noindent Now for $x\in \Z$ and $k\ge 1$, set
$$T_x(k) = \inf\{n\ge 0\ :\ Z_n(x)\ge k\}.$$ 
Denote also by $T_x:=T_x(1)$ the hitting time of $x$.  
In the following each time we consider an event of the type $\{T<T'\}$, for two random times $T$ and $T'$, we implicitely assume that 
it is contained in the set $\{T<\infty\}$. 
\begin{lem} 
\label{leminduction1}
Assume that there exists some $\alpha \in [0,1)$, such that $w_k\asymp (k+1)^\alpha$. 
Then there exist constants $c>0$ and $C>0$, such that for all $x\in \Z$, all $\gamma>1$, all $\epsilon\in (0,\alpha)$ and all $k \ge e^{C/\epsilon}$,      
$$\pp\left[T_{x+1}(k^\gamma) < T_x(k)\wedge T_{x+2}(k^{\frac{\gamma-1}{\alpha} +1-\epsilon}) \right]\le 
\exp\left(-c\, \frac{k^{1-\alpha}}{|\ln \epsilon|^{1/(1-\alpha)}}\right).$$
\end{lem} 
\begin{proof}
Let $\epsilon\in (0,\alpha)$ and $\gamma>1$ be given.
Consider the event 
$$A_0:=\left\{T_{x+1}(k^\gamma) < T_x(k)\wedge T_{x+2}(k^{\gamma'}) \right\},$$
where $\gamma':=(\gamma-1)/\alpha +1-\epsilon$. 
Set $K=[3\ln \epsilon/\ln \alpha]$.   
For $i=1,\dots,K$, set 
$$t_i:=T_{x+1}\left(\frac{k^\gamma}{(K-i+1)^2}\right),$$
and for $i\ge 2$,
$$N_i= \frac{k^\gamma}{(K-i+1)^2}-\frac{k^\gamma}{(K-i+2)^2}.$$
Set also $N_1=k^\gamma/K^2$. Note that by hypothesis, if $C>0$ is large enough, 
\begin{eqnarray}
\label{Ni}
N_i\ge k^\gamma/(K-i+2)^3\ge 1,
\end{eqnarray}
for all $i\le K$, and thus $t_i<t_{i+1}$. 
Moreover, since $w_k \asymp k^\alpha$,  
there exists some constant $c_0\in (0,1)$, such that for all $i_0<j_0$, all $i\ge i_0$ and all $j\le j_0$, $w_i/(w_i+w_j)\ge c_0i_0^\alpha/j_0^\alpha$. 
In particular before time $T_x(k)\wedge T_{x+2}(k^{\gamma'})$, at each visit to $x+1$, 
the probability to jump to $x$ is larger than $p_1:=c_0/k^{\alpha \gamma'}$. 
Thus if $t_1<T_x(k)\wedge T_{x+2}(k^{\gamma'})$, as it is the case on the event $A_0$ for instance, then 
the number of jumps from $x+1$ to $x$ before time $t_1+1$ stochastically dominates the sum of $N_1$ independent Bernoulli random variables with parameter $p_1$. 
Therefore,    
\begin{eqnarray*}
\pp\left[A_0,\ Z_{1+t_1}(x)\le N_1p_1/2\right] \le \exp(-c_1N_1p_1),
\end{eqnarray*}
for some constant $c_1>0$. 
Define next inductively $p_2,\dots,p_K$, and $A_1,\dots,A_{K+1}$, by 
$$p_{i} = c_0(N_{i-1}p_{i-1}/2)^\alpha k^{-\alpha\gamma'},$$
for $i\in \{2,\dots,K\}$, and 
$$A_{i}:=A_0\cap \left\{Z_{1+t_i}(x)\ge N_ip_i/2\right\},$$ 
for $i\in \{1,\dots,K+1\}$. 
Now by using the same argument as above, we immediately get by induction that 
\begin{eqnarray}
\label{Ai}
\pp\left[A_{i-1},\ Z_{1+t_i}(x)\le N_ip_i/2 \right] \le \exp\left(-c_1 N_ip_i\right),
\end{eqnarray}
for all $i\in \{2,\dots,K+1\}$. It is also straightforward to prove by induction, and by using \eqref{Ni}, that 
\begin{eqnarray}
\label{Nipi}
N_ip_i \ge \frac{2(c_0/2)^{1+\alpha+\dots+\alpha^i}}{\left(\prod_{j=1}^i (K-j+2)^{\alpha^{i-j}}\right)^3} k^{1-\alpha^i+\alpha\epsilon},
\end{eqnarray}
for all $i\le K$. On the other hand it is immediate that 
$$\sup_K \prod_{j=1}^K (K-j+2)^{\alpha^{K-j}}<\infty.$$
Thus there exists a constant $c'>0$ such that 
$$N_Kp_K \ge c' k^{1-\alpha^K+\alpha\epsilon}.$$
By taking now $\epsilon\ge C/\ln k$, with $C$ large enough, we deduce that $N_Kp_K/2>k+1$. 
Since on $A_0$, $t_K=T_{x+1}(k^\gamma)\le T_x(k)$, we get that $A_{K+1}$ is empty. 
Finally note that for all $i\le K$, 
$$\prod_{j=1}^i (K-j+2)^{\alpha^{i-j}}\le (K+1)^{1/(1-\alpha)},$$
so we also deduce from \eqref{Ai} and \eqref{Nipi} that   
\begin{eqnarray*}
\pp[A_0] &\le & \sum_{j=1}^{K+1} \pp\left[A_{i-1},\ Z_{1+t_i}(x)\le N_ip_i/2 \right] \\
&\le & (K+1)\exp(-c_2k^{1-\alpha}/|\ln \epsilon|^{1/(1-\alpha)}) \\
&\le &\exp(-ck^{1-\alpha}/|\ln \epsilon|^{1/(1-\alpha)}),
\end{eqnarray*} 
for some positive constants $c_2$ and $c$. This finishes the proof of the lemma. 
\end{proof} 
\noindent We can now give an alternative proof to Volkov's result, in the case when $w_k$ is of order $k^\alpha$. 
\begin{cor} 
\label{R'infini}
Assume that there exists $\alpha\in [0,1)$, such that $w_k\asymp (k+1)^\alpha$. Then a.s. $|R'|\in \{0,\infty\}$. 
\end{cor}
\begin{proof} 
Fix some $x\in \Z$ and some integers $N\ge 1$ and $z_0>0$. 
We want to prove that the event $\{Z_\infty(x)=\infty\}\cap \{Z_\infty(x-1)<z_0\}\cap \{Z_\infty(x+N)\le 1\}$ 
has probability zero, with the convention $Z_\infty(y):=\lim_{n\to \infty }Z_n(y)$, for all $y\in \Z$. 
 
\vspace{0.2cm}
\noindent For this first observe that for any $\epsilon<1/\alpha$, and any $m\ge 1$, 
\begin{eqnarray*}
\pp\left[ T_x(m)< T_{x+1}(m^{1/\alpha - \epsilon}) \wedge T_{x-1}(z_0)\right] &\le & \pp\left[\sum_{j=1}^m \xi_j \le z_0\right]
\le e^{-c(z_0)\ m^{\epsilon \alpha}}, 
\end{eqnarray*}
where $c(z_0)$ is some constant and the $\xi_j$'s are i.i.d. Bernoulli random variables with parameter $c'm^{-1+\epsilon\alpha}$, 
for some other constant $c'>0$.

\vspace{0.2cm}
\noindent Now define $\gamma_1,\dots, \gamma_N$, 
by $\gamma_1=1$, $\gamma_2=1/\alpha -\epsilon$, and for $i\ge 1$, 
$$\gamma_{i+2} = \gamma_i(1-\epsilon) +\frac{1}{\alpha}(\gamma_{i+1}-\gamma_i).$$
Note already, that if $\epsilon$ is small enough, then $\gamma_{i+1}>\gamma_i$, for all $i\le N-1$. 
Thus, as soon as $m$ is large enough, we can apply Lemma \ref{leminduction1} 
with $k=m^{\gamma_i}$ and $\gamma=\gamma_{i+1}/\gamma_i$, for any $i\in\{1,\dots,N-1\}$, and we get  
$$\pp\left[T_{x+i}(m^{\gamma_{i+1}})< T_{x+i-1}(m^{\gamma_i})\wedge T_{x+i+1}(m^{\gamma_{i+2}})\right]\le   
\exp\left(-c\, \frac{m^{\gamma_i(1-\alpha)}}{|\ln \epsilon|^{1/(1-\alpha)}}\right),
$$
where $c$ is the constant appearing in Lemma \ref{leminduction1}. Then, 
\begin{eqnarray*}
&& \pp\left[\{Z_\infty(x)=\infty\}\cap \{Z_\infty(x-1)<z_0\}\cap \{Z_\infty(x+N)\le 1\}\right] \\
& = & \pp\left[\cap_{m\to \infty}\, \{T_x(m)<T_{x-1}(z_0)\} \cap\{Z_\infty(x+N)\le 1\}\right]\\
&= &  \lim_{m\to \infty}\ \pp\left[\{T_x(m)<T_{x-1}(z_0)\} \cap\{Z_\infty(x+N)\le 1\}\right]\\
&\le & \lim_{m\to \infty}\ \Big\{ \pp\left[ T_x(m)< T_{x+1}(m^{1/\alpha - \epsilon}) \wedge T_{x-1}(z_0)\right]\\
&&  +\sum_{i=1}^{N-1} \pp\left[T_{x+i}(m^{\gamma_{i+1}})< T_{x+i-1}(m^{\gamma_i})\wedge T_{x+i+1}(m^{\gamma_{i+2}})\right]\Big\} \\
&=& 0,
\end{eqnarray*}
as wanted. Since this is true for any $x$, $N\ge 1$ and $z_0$, this proves the corollary. 
\end{proof}

\section{Proof of Theorem \ref{theo1}}
\label{alpha<1/2} 
\noindent We assume in this section that $\alpha <  1/2$.  

\vspace{0.2cm}
\noindent For $x\le 0$ and $m\ge 1$, consider the event 
$$E_x(m):=\left\{T_x(m)<T_{x-1}\right\}.$$ 
Then for $i\ge 1$, set 
$\epsilon_i = r/i^2$, with $r>0$ some positive constant which will be fixed later. 
Consider the sequence $(\gamma_i,i\ge 1)$ defined inductively by $\gamma_1 = 1$, $\gamma_2=(1/\alpha)-r$ and for $i\ge 1$, 
$$\gamma_{i+2} = \gamma_i(1-\epsilon_i) +\frac{1}{\alpha}(\gamma_{i+1}-\gamma_i).$$
Set 
$$F_x(m):=\left\{T_{x+i}(m^{\gamma_{i+1}}) < T_{x+i-1}(m^{\gamma_i})\quad \textrm{for all }i\ge 1\right\}.$$
Let us show that if $r$ is small enough, then
\begin{eqnarray}
\label{Hinduction1}
\sup_{x\le  0} \ \pp\left[F_x(m)^c \cap E_x(m)\right] =\kO\left(e^{-\kappa\,  m^{r\alpha}}\right),
\end{eqnarray}
as $m\to \infty$, for some constant $\kappa>0$. 
For this note that for all $i\ge 1$, 
$$\gamma_{i+2}-\gamma_{i+1} = (\frac 1\alpha-1)(\gamma_{i+1}-\gamma_i) -\epsilon_i\gamma_i,$$
so by induction we get
\begin{eqnarray}
\label{relrecgammai}
\gamma_{i+2}-\gamma_{i+1} = (\gamma_2-\gamma_1)(\frac 1 \alpha -1)^{i} - \sum_{j=1}^i \epsilon_j \gamma_j (\frac 1 \alpha -1)^{i-j}.
\end{eqnarray}
In particular $\gamma_{i+2}-\gamma_{i+1} \le (1/\alpha-1)^{i+1}$, for all $i\ge 1$, 
which implies $\gamma_i\le C_0(1/\alpha-1)^i$, for some constant $C_0>0$. 
Since $\sum1/i^2<\infty$, we see from \eqref{relrecgammai} that if $r>0$ is small enough, 
then there exists a constant $c_0>0$, such that 
$$\gamma_{i+2}-\gamma_{i+1}\ge c_0(\frac 1\alpha -1)^i.$$
Thus $\gamma_{i+2}\ge c_0(1/\alpha -1)^i$, for all $i\ge 1$, 
and since $\alpha<1/2$, $\gamma_i$ grows exponentially fast with $i$. Therefore, as soon as $m$ is large enough, 
we can apply Lemma \ref{leminduction1} with 
$k=m^{\gamma_i}$, $\gamma=\gamma_{i+1}/\gamma_i$ and $\epsilon=\epsilon_i$, for all $i\ge 1$. 
Then we get
\begin{eqnarray*}
\pp[F_x(m)^c\cap E_x(m)] &\le & \pp\left[T_x(m)< T_{x+1}(m^{\gamma_2}) \wedge T_{x-1}(z_0)\right] \\ 
& &+   \sum_{i\ge 1} \pp\left[T_{x+i}(m^{\gamma_{i+1}})< T_{x+i-1}(m^{\gamma_i})\wedge T_{x+i+1}(m^{\gamma_{i+2}})\right] \\
&\le &  e^{-\kappa\, m^{r \alpha}} 
+  \sum_{i\ge 1} \exp\left(-c\, \frac{m^{\gamma_i(1-\alpha)}}{|\ln \epsilon_i|^{1/(1-\alpha)}}\right),
\end{eqnarray*}
where $c$ is the constant appearing in Lemma \ref{leminduction1}, and $\kappa$ some other constant, see 
the proof of Corollary \ref{R'infini}. Since $\gamma_i$ grows exponentially fast with $i$, \eqref{Hinduction1} follows. 
But for any $x\le 0$ and any $m\ge 1$, the event $F_x(m)\cap E_x(m)$ is empty since $X$ 
cannot visit infinitely many sites in finite time.  
This proves that 
\begin{eqnarray}
\label{Exm} 
\sup_{x\le 0} \pp[E_x(m)] = \kO\left(e^{-\kappa \, m^{r\alpha}}\right),
\end{eqnarray}
as well, when $m\to \infty$. This proves in particular that for all $x\le 0$, $\pp[\cap_m E_x(m)]=0$. This means that a.s. the process 
cannot visit i.o. $x$ and never $x-1$. Actually the proof shows as well that for any $x\in \Z$, a.s. the process cannot visit
$x$ i.o. and only finitely many times $x-1$. Similarly, if $E'_x(m):=\{T_x(m)<T_{x+1}\}$, for $x\ge 0$, then 
\begin{eqnarray}
\label{E'xm} 
\sup_{x\ge 0} \pp[E'_x(m)] = \kO\left(e^{-\kappa \, m^{r\alpha}}\right),
\end{eqnarray}
as $m\to \infty$, and we can see that a.s. the process cannot visit i.o. some $x\in \Z$, and only finitely many times $x+1$.  
In other words, we just proved that a.s. either $R'=\Z$ or $R'=\emptyset$. 

\vspace{0.2cm}
\noindent Now observe that by using Lemma \ref{gene1}, if $\pp[X_n>0\ \forall n>0]=0$, then we know that $0$ is a.s. visited infinitely often. 
So with the result we just have proved, we know that in this case a.s. $R'=\Z$. So it only remains to consider the case 
when $\pp[X_n>0\ \forall n>0]>0$, which we assume now. We will prove that in this case the process is a.s. transient.  
To this end, note that \eqref{Exm} and \eqref{E'xm} show that 
$$\sum_{x\ge 0} \pp[\{T_x(x)<T_{x+1}\}\cup \{T_{-x}(x)<T_{-x-1}\}] < +\infty.$$
Thus according to Borel--Cantelli's lemma, a.s. for $x$ large enough, either $T_{x+1}<T_x(x)$ or $T_{-x-1}<T_{-x}(x)$.
For $n\ge 1$, denote by $x_n$ the $n$-th site visited by $X$, such that $T_{x_n}< T_{x_n-1}(x_n-1)$ and $x_n>0$,  
or $T_{x_n}<T_{x_n+1}(|x_n+1|)$ and $x_n<0$. Note that $|x_n|$ is at most of order $n$, 
so that for all $n$, if for instance $x_n> 0$, then $X$ has probability 
of order at least $n^{-\alpha}$ to jump to $x_n+1$ at time $T_{x_n}$, and similarly if $x_n<0$. 
Hence, 
$$\sum_{n\ge 1}\ \pp\left[T_{x_n+1} = T_{x_n}+1 \mbox{ or }  T_{x_n-1} = T_{x_n}+1 \mid \kF_{T_{x_n}}\right] = \infty.$$
It then follows from Levy's conditional Borel--Cantelli's lemma (see for instance Lemma 5.1 in \cite{T2}), that a.s. 
for infinitely many $n\ge 1$, either $T_{x_n+1} = T_{x_n}+1$ (if $x_n>0$) or $T_{x_n-1} = T_{x_n}+1$ (if $x_n<0$). 
But each  time this happens, by using our assumption we see that, independently of $\kF_{T_{x_n}}$, 
$X$ has some positive probability $p>0$ to never come back to $x_n$ after time $T_{x_n}$. It follows that a.s. 
this happens infinitely often, which proves well that $X$ is a.s. transient, as wanted.  \hfill $\square$

\vspace{0.2cm} 
\noindent \textit{Acknowledgments: I warmly thank an anonymous referee for having pointed out   
a serious mistake in a previous version of this paper.}

\end{document}